\documentclass[hidelinks,reqno]{amsart}
\usepackage{fullpage,amsmath,amsthm,amssymb,enumerate,mathtools,algorithm,algpseudocode,float,accents,chngcntr,hyperref}
\newtheorem{theorem}{Theorem}[section]
\newtheorem*{theorem*}{Theorem}
\newtheorem{lemma}[theorem]{Lemma}

\newtheorem{corollary}[theorem]{Corollary}

\theoremstyle{definition}
\newtheorem{definition}[theorem]{Definition}
\newtheorem{example}[theorem]{Example}
\newtheorem{notation}[theorem]{Notation}

\theoremstyle{remark}
\newtheorem{remark}[theorem]{Remark}

\DeclarePairedDelimiter\abs{\lvert}{\rvert}
\DeclarePairedDelimiter\norm{\lVert}{\rVert}
\DeclarePairedDelimiter\paren{(}{)}
\DeclarePairedDelimiter\braces{\lbrace}{\rbrace}
\DeclarePairedDelimiter\inprod{\langle}{\rangle}

\renewcommand{\P}{\mathbb{P}}

\newcommand{\E}{\mathbb{E}}
\newcommand{\R}{\mathbb{R}}
\newcommand{\C}{\mathbb{C}}

\newcommand{\tensor}{\otimes}

\newlength{\dhatheight}

\title{Energy optimization for distributions on the sphere and improvement to the Welch bounds}

\author{Yan Shuo Tan}
\address{Department of Mathematics, University of Michigan, 530 Church Street, Ann Arbor, MI 48109}
\email{yanshuo@umich.edu}
\date{\today}
\thanks{Partially supported by NSF Grant DMS 1265782 and U.S. Air Force Grant FA9550-14-1-0009.}

\begin{document}

\maketitle

\begin{abstract}
	For any Borel probability measure on $\R^n$, we may define a family of eccentricity tensors. This new notion, together with a tensorization trick, allows us to prove an energy minimization property for rotationally invariant probability measures. We use this theory to give a new proof of the Welch bounds, and to improve upon them for collections of real vectors. In addition, we are able to give elementary proofs for two theorems characterizing probability measures optimizing one-parameter families of energy integrals on the sphere. We are also able to explain why a phase transition occurs for optimizers of these two families.
\end{abstract}

\section{Introduction}

Amongst all Borel probability measures $\mu$ on $\R^n$ having the same radial distribution, we seek a minimizer for the energy integral
\begin{align}\label{integral}
I_k(\mu) := \int_{\R^n}\int_{\R^n}\inprod{x,y}^kd\mu(x)d\mu(y).
\end{align}
In this paper, we will introduce a tensorization trick, thereby proving that the integral is minimized by the rotationally invariant measure, $\mu_{rot}$. More precisely, for any integer $k$, we define the \emph{$k$-th eccentricity tensor} of a measure $\mu$. The gap between $I_k(\mu)$ and $I_k(\mu_{rot})$ is then given by the squared Euclidean norm of this tensor. Specializing to Borel probability measures on the sphere, we see that \eqref{integral} is minimized by the uniform measure. Moreover, we may also adapt the proof to obtain an analogous result for the uniform measure on the sphere in $\C^n$.

These facts have several interesting applications, the first of which concerns the well-known Welch bounds in the signal processing literature. Using the complex case of our result, we recover the original Welch bounds, while using the real case, we are able to improve upon them for collections of real vectors. In our opinion, this proof is more illuminating than the existing ones. It shows one view the Welch bounds as saying that the average cross-correlation of signal sets cannot beat that of the uniform distribution.

Next, we are able to obtain new proofs of Bj{\"o}rck's theorem from the 1950s and the recent theorem by Bilyk-Dai-Matzke. These theorems characterize optimizers of two one-parameter families of energy integrals, and were proved using methods from potential theory and spherical harmonics. Our methods have the benefit of being more elementary. Furthermore, our proof scheme for both theorems is very similar, and sheds light on the phase transition phenomenon discussed in \cite{Bilyk2016a}. Indeed, we are able to show why the phase transition occurs, and why it happens for different parameter values for the two families.

The plan of the rest of this paper is as follows. In Section 2, we define the eccentricity tensors and use the tensorization trick to prove the energy minimization property of rotationally invariant measures. In Section 3, we discuss the Welch bounds, show how they may be improved, and present some consequences of this improvement. In Section 4, we show how our results imply the two theorems on energy optimization on the sphere, and discuss their relevance to the phase transition phenomenon.

\section{Eccentricity tensors and the tensorization trick}

In this section, we shall introduce the tensorization trick, define eccentricity tensors, and prove that rotationally invariant measures minimize \eqref{integral}. For notational as well as intuition purposes, however, it is more convenient to work with random vectors than with measures. We hence do so for the rest of this paper, being careful to assert the independence of collections of random vectors where necessary.

The tensorization trick is to write the integral \eqref{integral} as the squared Euclidean norm of the $k$-th moment tensor of $\mu$.

\begin{notation}
	Let $X$ be a random vector in $\R^n$. For any positive integer $k$, let
	\[
	M^k_X := \E{X^{\tensor k}}
	\]
	denote its $k$-th moment tensor if all entries are finite.
\end{notation}

Recall the following fact from linear algebra. For any positive integer $k$, we may identify the $k$-th tensor product $T^k(\R^n) = \R^n\tensor\cdots\tensor\R^n$ with $\R^{n^k}$ by picking as a basis the vectors $\braces{e_{i_1}\tensor e_{i_2}\tensor\cdots\tensor e_{i_k}}_{1\leq i_1,\ldots i_k \leq n}$. With this choice, the Euclidean inner product between any two pure tensors $u_1\tensor\cdots\tensor u_k$ and $v_1\tensor\cdots\tensor v_k$ can be written as
\[
\inprod{u_1\tensor\cdots\tensor u_k,v_1\tensor\cdots\tensor v_k} = \prod_{i=1}^k\inprod{u_i,v_i}.
\]
In particular, for power tensors $u^{\tensor k}$ and $v^{\tensor k}$, we have the formula
\begin{align} \label{eq: tensorizaton trick for power tensors}
\inprod{u^{\tensor k},v^{\tensor k}} = \inprod{u,v}^k.
\end{align}

Now if $X$ and $Y$ are two independent random vectors, we may rewrite the $k$-th moment of their inner product as an inner product between their $k$-th moment tensors. Namely, we have
\begin{align} \label{tensorid1}
\E\paren{\inprod{X,Y}^k} = \E{\inprod{X^{\tensor k},Y^{\tensor k}}} = \inprod*{\E{X^{\tensor k}},\E{Y^{\tensor k}}} = \inprod*{M^k_X,M^k_Y},
\end{align}
where the first equality follows from equation \eqref{eq: tensorizaton trick for power tensors}. For independent copies $X$ and $X'$ of the same random vector having distribution $\mu$, $M^k_X = M^k_{X'}$, so
\begin{align} \label{tensorid2}
I_k(\mu) = \E\paren{\inprod{X,X'}^k} = \norm{M^k_X}^2.
\end{align}
Here and in the rest of this paper, we will use $\norm{\cdot}$ to denote the vector Euclidean norm. No other norms are used, so there should be no risk of confusion.

We next introduce the notion of the rotation symmetrization of a random vector.

\begin{definition}
	For any random vector $X$ in $\R^n$, let $X_{rot}$ denote a random vector that is independent of $X$, has the same radial distribution as $X$, and whose distribution is rotationally invariant (i.e. $QX_{rot} \stackrel{d}{=} X_{rot}$ for all $Q \in O(n)$). We call $X_{rot}$ the \emph{rotational symmetrization} of $X$.
\end{definition}

Comparing the moment tensors of a random vector and those of its rotational symmetrization give rise to what we shall call eccentricity tensors.

\begin{definition}
	Let $X$ be a random vector in $\R^n$ with finite moments of all orders. For any positive integer $k$, define its \emph{$k$-th eccentricity tensor} to be
	\begin{align}
	E^k_X := M^k_X - M^k_{X_{rot}}.
	\end{align} 
\end{definition}

Since $X \stackrel{d}{=} X_{rot}$ if and only if $X$ is rotationally invariant, we see that the eccentricity tensors of $X$ are quantitative measures of how far its distribution is from being rotationally invariant. This interpretation is further supported by the following observation.

\begin{lemma}[Orthogonality] \label{lem: orthogonality}
	Let $X$ be a random vector in $\R^n$ with finite moments of all orders. Its eccentricity tensors are orthogonal to the moment tensors of its rotational symmetrization. In other words, for any positive integer $k$,
	\begin{align} \label{eq: orth 1}
	\inprod{E^k_X, M^k_{X_{rot}}} = 0
	\end{align}
	and
	\begin{align} \label{orthogonality}
	\norm*{M^k_X}^2 = \norm*{M^k_{X_{rot}}}^2 + \norm*{E^k_X}^2.
	\end{align}
\end{lemma}

\begin{proof}
	Let $Q$ be a random orthogonal matrix chosen according to the Haar measure on $O(n)$. For any fixed vector $v \in \R^n$, $Qv$ is uniformly distributed on the sphere of radius $\norm{v}$, so if $Y$ is any random vector independent of $Q$, applying $Q$ to $Y$ preserves its radial distribution but makes $QY$ rotationally invariant.
	
	Now choose $Q$ to be independent of $X$ and $X_{rot}$. Our previous discussion implies that
	\[
	Q^TX \stackrel{d}{=} X_{rot} \stackrel{d}{=} QX_{rot}.
	\]
	We use this to compute
	\begin{align}
	\E\paren{\inprod{X,X_{rot}}^k} = \E\paren{\inprod{X,QX_{rot}}^k} = \E\paren{\inprod{Q^T X,X_{rot}}^k} = \E\paren{\inprod{X_{rot}',X_{rot}}^k},
	\end{align}
	where $X_{rot}'$ is an independent copy of $X_{rot}$. We may then apply identities \eqref{tensorid1} and \eqref{tensorid2} to rewrite the above equation as
	\begin{align}
	\inprod*{M^k_X,M^k_{X_{rot}}} = \inprod*{M^k_{X_{rot}},M^k_{X_{rot}}}.
	\end{align}
	Subtracting the right hand side from the left hand side gives \eqref{eq: orth 1}, from which \eqref{orthogonality} is an immediate corollary.
\end{proof}

The fact that the integral \eqref{integral} is minimized by rotationally invariant measures is then an easy consequence of the previous lemma. To show that these are the \emph{unique} minimizers, we need further assumptions on our random vectors to ensure that they are determined by their moment tensors. A sufficient condition is that of being subexponential.\footnote{This is a weak condition satisfied by most distributions dealt with in practice. For an in-depth discussion on the properties of subexponential distributions, we refer the reader to \cite{Vershynin2011a}.}

\begin{definition}
	We say that a real random variable $X$ is \emph{subexponential} if it has exponential tail decay, i.e. if there is some $K$ such that for all $t \geq 0$,
	$$
	\P\paren{\abs{X} > t} \leq \exp(1 - t/K).
	$$
	We say that a random vector $X$ in $\R^n$ is subexponential if all its one-dimensional marginals are subexponential. Equivalently, it is subexponential if its radial distribution is subexponential.
\end{definition}

\begin{lemma} \label{lem: subexponential RV determined by moment tensors}
	Let $X$ be a subexponential random vector in $\R^n$. Then the distribution of $X$ is determined by its moment tensors.
\end{lemma}

\begin{proof}
	By the definition of being subexponential, we have the following moment growth condition \cite{Vershynin2011a}:
	\begin{align} \label{moments}
	\sup_{v \in S^{n-1}}\limsup_{r \to \infty} \frac{\paren*{\E{\abs*{\inprod{X,v}}^r}}^{1/r}}{r} < \infty.
	\end{align}
	Let $\phi_X(v) = \E{e^{i\inprod{X,v}}}$ denote the characteristic function of $X$. The above condition implies that for each $v \in S^{n-1}$, the function $t \mapsto \E{e^{it\inprod{X,v}}}$ can be written as a power series with coefficients $\frac{\E{\inprod{X,v}^r}}{r!}$ \cite{Billingsley1995a}, so $\phi_X(v)$ is determined by the moments $\E{\inprod{X,v}^r}$. By \eqref{tensorid1}, $\E{\inprod{X,v}^r} = \inprod{M_X^r,v^{\tensor r}}$, so these are functions of the moment tensors. Finally, it is a fact from elementary probability that a random vector in $\R^n$ determined by its characteristic function (see exercise 2.36 in \cite{Cnlar2011a}).
\end{proof}

We can thus summarize our results so far in the following theorem.

\begin{theorem} \label{thm: minimization and uniqueness theorem}
	Let $X$ be a random vector in $\R^n$. Then
	\begin{enumerate}
		\item[a)] (Minimization) If $X'$ is an independent copy of $X$, and $X_{rot}, X_{rot}'$ are independent copies of its rotational symmetrization, we have
		\begin{align} \label{minimization}
		\E\paren{\inprod{X,X'}^k} \geq \E\paren{\inprod{X_{rot},X_{rot}'}^k}
		\end{align}
		for any positive integer $k$ so long as $X$ has finite $k$-th moment.
		\item[b)] (Uniqueness) Furthermore, if equality holds in \eqref{minimization} for all $k$ and we assume that $X$ has a subexponential distribution, then $X$ is rotationally invariant.
	\end{enumerate}
\end{theorem}

\begin{proof}
	Using identity \eqref{tensorid2}, we rewrite the first claim as
	\[
	\norm*{M^k_X}^2 \geq \norm*{M^k_{X_{rot}}}^2,
	\]
	and this follows immediately from equation \eqref{orthogonality}.
	
	If equality holds for all positive integers $k$, then by \eqref{orthogonality}, $E^k_X = 0$ for all $k$, implying that $X$ and $X_{rot}$ have the same moment tensors of all orders. If we assume that $X$ is subexponential, Lemma \ref{lem: subexponential RV determined by moment tensors} implies that $X$ and $X_{rot}$ have the same distribution.
\end{proof}

For the remainder of this paper, we specialize to the case of distributions on the sphere.

\begin{corollary} \label{cor: spherical minimization}
	Let $\theta$ have the uniform distribution on the sphere $S^{n-1}$, and let $X$ be any random vector taking values on the sphere. Let $\theta'$ and $X'$ be independent copies of $\theta$ and $X$ respectively. Then
	\begin{align}
	\E\paren{\inprod{X,X'}^{2k}} \geq \E\paren{\inprod{\theta,\theta'}^{2k}} = \frac{1\cdot 3\cdots(2k-1)}{n\cdot(n+2)\cdots(n+2k-2)}
	\end{align}
	for any positive integer $k$. Furthermore, if equality holds for all $k$, $X$ has the uniform distribution.
\end{corollary}

\begin{proof}
	Clearly $\theta \stackrel{d}{=} X_{rot}$, and is subexponential. The inequality and the characterization statement then follows immediately from Theorem \ref{thm: minimization and uniqueness theorem}. By uniformity, we have $\E\paren{\inprod{\theta,\theta'}^{2k}} = \E\paren{\inprod{\theta,v}^{2k}}$ for any unit vector $v \in S^{n-1}$, and the explicit computation for $\E\paren{\inprod{\theta,v}^{2k}}$ is the content of the next lemma.
\end{proof}

\begin{lemma}[Moments of spherical marginals] \label{lem: moments of spherical marginals}
	Let $\theta$ be uniformly distributed on the sphere $S^{n-1}$. Then for any unit vector $v \in S^{n-1}$ and any positive integer $k$, we have
	\begin{align} \label{spheremoment}
	\E\paren{\inprod{\theta,v}^{2k}} = \frac{1\cdot 3\cdots(2k-1)}{n\cdot(n+2)\cdots(n+2k-2)}
	\end{align}
\end{lemma}

\begin{proof}
	This is a standard calculation that we include for completeness. We shall prove this by computing gaussian integrals. Let $\gamma$ and $g$ denote standard gaussians in 1 dimension and $n$ dimensions respectively. Using the rotational invariance of $g$, we have
	\[
	\E{\gamma^{2k}} = \E\paren{\inprod{g,v}^{2k}} = \E\paren{\inprod{\norm{g}\theta,v}^{2k}} = \E{\norm{g}^{2k}}\E\paren{\inprod{\theta,v}^{2k}}.
	\]
	Rearranging gives
	\[
	\E{\inprod{\theta,v}^{2k}} = \frac{\E{\gamma^{2k}}}{\E{\norm{g}^{2k}}}.
	\]
	We then compute
	\begin{align} \label{gaussiannorm}
	\E{\norm{g}^{2k}} = \frac{\omega_n}{(2\pi n)^{n/2}}\int_0^\infty r^{2k}r^{n-1}e^{-r^2/2}dr,
	\end{align}
	where $\omega_n$ is the volume of the sphere $S^{n-1}$. It is well known that
	\[
	\omega_n = \frac{2\pi^{n/2}}{\Gamma(n/2)},
	\]
	while we also have
	\[
	\int_0^\infty r^{2k}r^{n-1}e^{-r^2/2}dr = 2^{n/2+k-1}\Gamma(n/2+k).
	\]
	Substituting these back into \eqref{gaussiannorm} gives
	\[
	\E{\norm{g}^{2k}} = 2^k\frac{\Gamma(n/2+k)}{\Gamma(n/2)} = n\cdot(n+2)\cdots(n+2k-2).
	\]
	This yields the denominator in \eqref{spheremoment}. A similar calculation for $\E{\gamma^{2k}}$ yields the numerator. 
\end{proof}

\section{Applications to dictionary incoherence and the Welch bounds}

Given a collection of $m$ unit vectors $Z = \braces{z_1,z_2,\ldots,z_m}$ in $\C^n$, we are often interested in the quantity
\[
c_{max} = \max_{i \neq j}\abs*{\inprod*{z_i,z_j}}.
\]
If we think of the vectors as dictionary elements, then $c_{max}$ measures the coherence or maximum cross-correlation of the dictionary. It is well known in the sparse approximation literature that the larger the value of $c_{max}$, the worse the collection $Z$ performs when we try to recover a sparse representation of a vector as a linear combination of the $z_j$'s \cite{Donoho2003}. As such, it is an important question in the design of communication systems to know how well we can do theoretically, and how we may find collections that achieve the theoretical minimum value of $c_{max}$.

In 1974, Welch gave a family of lower bounds on $c_{max}$ in terms of $m$ and $n$.

\begin{theorem}[Welch, 1974 \cite{Welch1974a}]
	Let $Z$ and $c_{max}$ be defined as above. Then for each positive integer $k$, we have
	\begin{align} \label{welch1}
	\paren*{c_{max}}^{2k} \geq \frac{1}{m-1}\paren*{\frac{m}{\binom{n+k-1}{k}}-1}.
	\end{align}
\end{theorem}

Welch proved this theorem by bounding the average cross-correlation (also sometimes called the $p$-frame potential, with $p=2k$ \cite{Ehler2012}).

\begin{lemma}[Welch]
	Let $\braces{z_1,z_2,\ldots,z_m}$ be unit vectors in $\C^n$, then
	\begin{align} \label{welch2}
	\frac{1}{m^2}\sum_{i,j=1}^m \abs{\inprod{z_i,z_j}}^{2k} \geq \binom{n+k-1}{k}^{-1}.
	\end{align}
\end{lemma}

By separating the diagonal terms from the sum and rearranging the summands, it is easy to see how \eqref{welch2} implies \eqref{welch1}. Welch's original proof of \eqref{welch2} was combinatorial in nature. In 2003, Alon \cite{Alon2003} provided a geometric proof based on examining the Gram matrix associated to $Z$ and dimension counting. The proof was reproduced by Datta et al. \cite{Datta2012a} in 2012, who were apparently unaware of the earlier paper.

Both arguments are agnostic to whether the vectors are real or complex, and it is a natural question whether one may improve the bound when we restrict to the case of real vectors. Using the energy minimization property of rotationally invariant distributions, we are able to show that this is indeed the case.

\begin{lemma}
	Let $\braces{x_1,x_2,\ldots,x_m}$ be unit vectors in $\R^n$. Then
	\begin{align} \label{newwelch}
	\frac{1}{m^2}\sum_{i,j=1}^m \abs{\inprod{x_i,x_j}}^{2k} \geq \frac{1\cdot 3\cdots(2k-1)}{n\cdot(n+2)\cdots(n+2k-2)}.
	\end{align}
\end{lemma}

\begin{remark}
	Since
	\[
	\binom{n+k-1}{k}^{-1} = \frac{1\cdot2\cdots k}{n\cdot(n+1)\cdots (n+k-1)},
	\]
	we see that the new bound \eqref{newwelch} is equal to the old one \eqref{welch2} for $k=1$, and is strictly larger for $k >1$.
\end{remark}

\begin{proof}
	Let $X$ be uniformly distributed on the set $\braces{x_1,x_2,\ldots,x_m}$. Corollary \ref{cor: spherical minimization} applies and we have
	\[
	\E\paren{\inprod{X,X'}^{2k}} \geq \frac{1\cdot 3\cdots(2k-1)}{n\cdot(n+2)\cdots(n+2k-2)}
	\]
	for any positive integer $k$. On the other hand, we also have
	\begin{equation*}
	\E\paren{\inprod{X,X'}^{2k}} = \frac{1}{m^2}\sum_{i,j=1}^m \abs{\inprod{x_i,x_j}}^{2k}. \qedhere
	\end{equation*}
\end{proof}

\begin{remark}
	In an earlier version of this paper, we stated that this result is new. I have since found it stated in \cite{Ehler2012} by Ehler and Okoudjou, who attribute it to Venkov \cite{Venkov}. The proof in \cite{Ehler2012}, however, proceeds via spherical harmonics and not the tensorization machinery we have used here.
\end{remark}

Let us illustrate the improved bound by revisiting an example from \cite{Datta2012a}.

\begin{example}
	Let $x_1,x_2,\ldots, x_7$ be the columns of
	\[
	\begin{bmatrix}
	0.99 & 0.14 & 0.56 & -0.68 & 0.93 & -0.86 & 0.30 \\
	0.08 & 0.99 & 0.83 & 0.73 & -0.36 & -0.50 & 0.95
	\end{bmatrix}.
	\]
	This collection achieves the $k=1$ Welch bound \eqref{welch2}, and its energy\footnote{To compute this value, we renormalized the vectors $x_1,\ldots,x_7$ in order to reduce roundoff error.}
	\[
	\sum_{i,j=1}^7 \abs{\inprod{x_i,x_j}}^6 = 15.3128
	\]
	was experimentally observed to be minimal over all collections of $7$ unit vectors in $\R^2$. However, the $k=3$ Welch bound gives a lower bound of $12.25$ for the energy, so there was a gap between theory and experiment. Using our improved bound \eqref{newwelch}, we get 15.3125, thereby bridging this gap completely.
\end{example}

Although the improved bounds do not hold for complex collections of vectors, we are nonetheless able to recover the original Welch bounds using the same circle of ideas and making a few adjustments.

\begin{definition}
	For any random vector $X$ in $\C^n$, let $X_{uni}$ denote a random vector that is independent of $X$, has the same radial distribution as $X$, and whose distribution is invariant under unitary transformations. We call $X_{uni}$ the \emph{unitary symmetrization} of $X$.
\end{definition}

With this definition, we can state the following complex version of Theorem \ref{thm: minimization and uniqueness theorem}.

\begin{theorem}
	Let $X$ be a random vector in $\C^n$ with finite moments of all orders. Then if $X'$ is an independent copy of $X$, and $X_{uni}, X_{uni}'$ are independent copies of its unitary symmetrization, we have
	\begin{align} \label{min2}
		\E{\abs*{\inprod{X,X'}}^{2k}} \geq \E{\abs*{\inprod{X_{uni},X_{uni}'}}^{2k}}
	\end{align}
	for any positive integer $k$.
\end{theorem}

\begin{proof}
	By considering the moment tensors
	\[
	M^{2k}_X := \E{X^{\tensor k}\tensor (X^*)^{\tensor k}},
	\]
	we may define a complex version of eccentricity tensors. Next, we replace $Q \sim \text{Haar}\paren*{O(n)}$ with $U \sim \text{Haar}\paren*{U(n)}$ in Lemma \ref{lem: orthogonality} to prove an orthogonality result analogous to \eqref{orthogonality}. With this result, \eqref{min2} follows immediately.
\end{proof}

We are now able to complete the proof of \eqref{welch2} with the help of the following version of Lemma \ref{lem: moments of spherical marginals}.

\begin{lemma}[Moments of complex spherical marginals]
	Let $\theta$ be uniformly distributed on the complex sphere $S^{2n-1} \subset \C^n$. Then for any unit vector $v \in S^{2n-1}$ and any positive integer $k$, we have
	\begin{align} \label{cxspheremoment}
	\E{\abs*{\inprod{\theta,v}}^{2k}} = \binom{n+k-1}{k}^{-1}.
	\end{align}
\end{lemma}

\begin{proof}
	Let $\gamma$ and $g$ denote standard complex gaussians in 1 dimension and $n$ dimensions respectively. Then
	\[
	\E{\abs*{\inprod{\theta,v}}^{2k}} = \frac{\E{ \abs*{\gamma}^{2k}}}{\E{ \norm{g}^{2k}}}.
	\]
	Since $\abs*{\gamma}$ is the norm of a two-dimensional standard real gaussian, while $\norm{g}$ is the norm of a $2n$-dimensional standard real gaussian, \eqref{cxspheremoment} follows from the calculations of gaussian integrals done in Lemma \ref{lem: moments of spherical marginals}.
\end{proof}

\begin{remark}
	Given $Z = \braces{z_1,z_2,\ldots,z_m}$ a set of unit vectors in a Hilbert space $\mathbb{H}$, $k$ a positive integer, define the set
	\[
	Z^{(k)} = \braces{z_1^{\tensor k},z_2^{\tensor k},\ldots,z_m^{\tensor k}} \subset \text{Sym}^k\paren*{\mathbb{H}}.
	\]
	Datta et al.'s paper \cite{Datta2012a} characterized sets $Z$ achieving equality in the $k$-th Welch average cross-correlation bound \eqref{welch2} as those for which $Z^{(k)}$ forms a tight frame for $\text{Sym}^k\paren*{\mathbb{H}}$. Since our results show that this bound is not tight when $\mathbb{H}$ is a real Hilbert space and $k >1$, we have proved that tight frames of the form $Z^{(k)}$ do not exist for symmetric spaces of real tensors with $k >1$. Indeed, this also holds true for generalized frames as defined by the same authors.
\end{remark}

\begin{remark}
	Datta et al. \cite{Datta2012a} showed that the analogous statement for complex vector spaces is false. In fact, if $\theta$ is distributed uniformly on the complex sphere $S^{2n-1} \subset \C^n$, then
	\[
	\E{ \theta^{\tensor k} \tensor (\theta^*)^{\tensor k} } = \binom{n+k-1}{k}^{-1} I_{\text{Sym}^k(\C^n)}.
	\]
\end{remark}

\section{Applications to energy optimization on the sphere}

In two recent papers \cite{Bilyk2016,Bilyk2016a}, Bilyk et al. presented a theorem characterizing probability measures minimizing geodesic distance energy integrals. This is an analogue of Bj{\"o}rck's theorem from 1956 which characterized probability measures minimizing energy integrals based on Euclidean distance \cite{Bjorck1956a}. Bj{\"o}rck proved his theorem by considering Riesz potentials, while Bilyk et al. proved their result using spherical harmonic expansions and the hermisphere Stolarsky principle. In this section, we show how to derive both results using the tensorization trick and the energy minimization property of the uniform distribution on the sphere.

\begin{theorem}[Bilyk-Dai-Matzke, 2016] \label{thm: Bilyk theorem}
	For $\delta > 0$, define the geodesic energy integral
	\begin{align} \label{geodenergy}
	G_\delta(\mu) := \int_{S^{n-1}}\int_{S^{n-1}} d(x,y)^\delta d\mu(x) d\mu(y),
	\end{align}
	where $d(x,y)$ denotes the geodesic distance between $x$ and $y$. The maximizers of this energy integral over Borel probability measures on $S^{n-1}$ can be characterized as follows:
	\begin{enumerate}
		\item[a)] $0 < \delta < 1$: the unique maximizer of $G_\delta(\mu)$ is $\mu = \sigma$, the uniform measure.
		\item[b)] $\delta = 1$: $G_\delta(\mu)$ is maximized if and only if $\mu$ is centrally symmetric.
		\item[c)] $\delta > 1$: $G_\delta(\mu)$ is maximized if and only if $\mu = \frac{1}{2}(\delta_p + \delta_{-p})$, i.e. the mass is supported equally by two antipodal points.
	\end{enumerate}
\end{theorem}

\begin{proof}
	Observe that the geodesic distance $d(x,y)$ is simply the angle between $x$ and $y$. As such, we have $d(x,y) = \arccos\paren*{\inprod{x,y}}$. We may thus rewrite \eqref{geodenergy} as 
	\[
	G_\delta(\mu) = \E{ \arccos\paren*{\inprod{X,X'}}^\delta},
	\]
	where $X$ and $X'$ are independent random vectors with distribution $\mu$.
	
	Let us start by proving part b). It is an exercise to show that the even derivatives of $\arccos$ vanish at $0$, while the odd derivatives are strictly negative at $0$. For $-1 < t < 1$ may hence write $\arccos$ as its Taylor series
	\begin{align}
	\arccos\paren*{t} = \frac{\pi}{2} - \sum_{k=0}^\infty a_{2k+1}t^{2k+1}
	\end{align}
	where $a_{2k+1} > 0$ for all $k$. We claim that in fact, the above formula holds for all $t$ in the \emph{closed} interval $[-1,1]$, and furthermore that the series is absolutely convergent. This is the content of Lemma \ref{lem: Taylor series absolute convergence on closed interval} to come. As a result, we may use Fubini to interchange sums and expectations, thereby writing
	\begin{align*}
	\E{ \arccos\paren*{\inprod{X,X'}}} = \frac{\pi}{2} - \sum_{k=0}^\infty a_{2k+1}\E{ \inprod{X,X'}^{2k+1}}.
	\end{align*}
	
	Since $\E{ \inprod{X,X'}^{2k+1}} \geq 0$ for each $k$ by identity \eqref{tensorid2}, this last expression is maximized if and only if $\E{ \inprod{X,X'}^{2k+1}} = 0$ for every non-negative integer $k$. By the same identity, this happens if and only if all odd moments of $X$ are zero, i.e. if and only if $X$ is centrally symmetric. This proves the case $\delta = 1$.
	
	Now let $0 < \delta < 1$. We claim that for $-1 \leq t \leq 1$, we may write
	\begin{align}
	\arccos(t)^\delta = \paren*{\frac{\pi}{2}}^\delta - \sum_{k=1}^\infty a_kt^k
	\end{align} 
	where $a_k > 0$ for all $k > 0$, and that the series is absolutely convergent. Lemma \ref{lem: sign of derivatives} (to come) tells us that the Taylor series of $\arccos(t)^\delta$ has this form, which combined with Lemma \ref{lem: Taylor series absolute convergence on closed interval} proves this claim. As such, we may again use Fubini to write
	\begin{align}
	\E{ \arccos\paren*{\inprod{X,X'}}^\delta} = \paren*{\frac{\pi}{2}}^\delta - \sum_{k=1}^\infty a_k\E{ \inprod{X,X'}^k}.
	\end{align}
	By identity \eqref{tensorid2}, $\E{\inprod{X,X'}^k} \geq 0$ for any distribution, while by Corollary \ref{cor: spherical minimization}, the uniform measure uniquely minimizes all of these moments simultaneously. As such, we see that it is the unique maximizer of $G_\delta(\mu)$.
	
	The remaining case where $\delta > 1$ is easy and does not require a proof using our methods. For completeness, we repeat the proof given by the original authors \cite{Bilyk2016a}. Since $d(x,y) \leq \frac{\pi}{2}$, we have
	\[
	G_\delta(\mu) \leq \paren*{\frac{\pi}{2}}^{\delta -1}\int_{S^{n-1}}\int_{S^{n-1}} d(x,y) d\mu(x) d\mu(y) \leq \paren*{\frac{\pi}{2}}^{\delta}.
	\]
	The first inequality is tight whenever $d(x,y)$ only takes the values $\frac{\pi}{2}$ and $0$, while by part b), the second inequality becomes equality when $\mu$ is centrally symmetric. Together, these imply that $\mu = \frac{1}{2}(\delta_p + \delta_{-p})$ for some $p \in S^{n-1}$.
\end{proof}

\begin{lemma} \label{lem: Taylor series absolute convergence on closed interval}
	Let $f$ be a function that is continuous on $[-1,1]$ and that agrees with its Taylor series at $0$ on the open interval $(-1,1)$. Suppose further that all but finitely many of its derivatives at $0$ have the same sign. Then the series is absolutely convergent over the closed interval $[-1,1]$, and agrees with $f$ over the interval.
\end{lemma}

\begin{proof}
	By subtracting off polynomials and negating the function if necessary, we may assume without loss of generality that the Taylor series for $f(t)$ is given by $\sum_{k=0}^\infty c_k t^k$ where $c_k \geq 0$ for all $k$. By the monotone convergence theorem, together with our assumptions on $f$, we have
	\[
	\sum_{k=0}^\infty c_k = \lim_{t \to 1^-}\sum_{k=0}^\infty c_k t^k = \lim_{t \to 1^-} f(t) = f(1).
	\]
	As such, the series $\sum_{k=0}^\infty c_k$ is absolutely convergent, and the Taylor series is also absolutely convergent on the closed interval $[-1,1]$. Finally, we can apply the dominated convergence theorem to see that $f(-1) = \sum_{k=0}^\infty c_k(-1)^k$.
\end{proof}

\begin{lemma} \label{lem: sign of derivatives}
	Let $f$ be a function that has derivatives of all orders at $0$ and let $0 < \alpha < 1$. Suppose $f(0) > 0$ and $f'(0) < 0$, while all higher derivatives $f$ at $0$ are non-positive, then all derivatives of $f^\alpha$ at $0$ are strictly negative.
\end{lemma}

\begin{proof}
	Let $F(t) = f(t)^\alpha$. By induction, one may observe that for any positive integer $k$, $F^{(k)}(t)$ is a sum of $2^{k-1}$ terms of the form
	\[
	g_{\vec{n}}(t) := f(t)^{\alpha-j} \paren*{\prod_{i=0}^{j-1}\paren*{\alpha-i} } \paren*{\prod_{i=0}^{j-1} f^{(n_i)}(t)},
	\]
	where $1 \leq j \leq k$, and $\vec{n} = \paren*{n_0,n_1,\ldots,n_{j-1}}$ is a vector of positive integers. If there is some index $i$ such that $f^{(n_i)}(0) = 0$, then $g_{\vec{n}}(0) = 0$. Otherwise, $\prod_{i=0}^{j-1} f^{(n_i)}(0)$ is a product of $j$ negative numbers and so has sign $(-1)^j$. On the other hand, our assumption on $\alpha$ imply that $\paren*{\prod_{i=0}^{j-1}\paren*{\alpha-i}}$ is a product of one positive number and $j-1$ negative numbers, and so has sign $(-1)^{j-1}$. As such, $g_{\vec{n}}(0) \leq 0$.
	
	Finally, notice that $F^{(k)}(0)$ always contains the term
	\[
	g_{(1,1,\ldots,1)}(0) = f(t)^{\alpha-k}\paren*{\prod_{i=0}^{k-1}\paren*{\alpha-i} } f'(0)^k.
	\]
	Since we have assumed that $f'(0) < 0$, this term is strictly negative. As such, $F^{(k)}(0)$ is also negative, as was to be shown.
\end{proof}

In the course of proving the previous theorem, we have in fact proved the following more general result.

\begin{theorem} \label{thm: energy minimization for general functions}
	Let $F$ be a function on on $[-1,1]$ that is given by the power series
	\begin{align}
	F(t) = a_0 - \sum_{k=1}^\infty a_k t^k,
	\end{align}
	where $a_k \geq 0$ for all $ k > 0$. Then the energy integral
	\begin{align} \label{energyintegral}
	I_F(\mu) := \int_{S^{n-1}}\int_{S^{n-1}}F(\inprod{x,y}) d\mu(x) d\mu(y)
	\end{align}
	is maximized over all Borel probability measures on $S^{n-1}$ by the uniform measure. Furthermore, if $a_k > 0$ for all $k > 0$, then the maximizer is unique.
\end{theorem}

Let us see how we may apply this more general theorem to recover Bj{\"o}rck's original result.

\begin{theorem}[Bj{\"o}rck, 1956]
	For $\delta > 0$, define the energy integral
	\begin{align} \label{eucenergy}
	E_\delta(\mu) = \int_{S^{n-1}}\int_{S^{n-1}} \norm{x-y}^\delta d\mu(x) d\mu(y).
	\end{align}
	The maximizers of this energy integral over Borel probability measures on $S^{n-1}$ can be characterized as follows:
	\begin{enumerate}
		\item $0 < \delta < 2$: the unique maximizer of $E_\delta(\mu)$ is $\mu = \sigma$, the uniform measure.
		\item $\delta = 2$: $E_\delta(\mu)$ is maximized if and only if the center of mass of $\mu$ is at the origin.
		\item $\delta > 2$: $E_\delta(\mu)$ is maximized if and only if $\mu = \frac{1}{2}(\delta_p + \delta_{-p})$, i.e. the mass is supported equally by two antipodal points.
	\end{enumerate}
\end{theorem}

\begin{proof}
	We rewrite \eqref{eucenergy} as
	\[
	E_\delta(\mu) = \E{\norm{X-X'}^\delta}
	\]
	where $X$ and $X'$ are independent random vectors with distribution $\mu$. The easy case $\delta > 2$ is proved exactly as in Theorem \ref{thm: Bilyk theorem}. The case $\delta=2$ is also clear, for we may write $\norm{X-X'}^2 = 2 - 2\inprod{X,X'}$, and by identity \eqref{tensorid2},
	\[
	E_2(\mu) = 2 - \E{\inprod{X,X'}} = 2 - \norm{\E{X}}^2.
	\]
	This is maximized if and only if $\E{X} = 0$.
	
	For $0 < \delta < 2$, we set $f(t) = 2-2t$ and $F(t) = f(t)^{\delta/2}$. Then $f$ and $\alpha = \delta/2$ satisfy the hypotheses of Lemma \ref{lem: sign of derivatives}, so $F^{(k)}(0) < 0$ for all positive integers $k$. This, together with Lemma \ref{lem: Taylor series absolute convergence on closed interval}, implies that $F$ satisfies the hypothesis of Theorem \ref{thm: energy minimization for general functions}. Since
	\[
	E_\delta(\mu) = \int_{S^{n-1}}\int_{S^{n-1}}\paren*{2-2\inprod{x,y}}^{\delta/2} d\mu(x) d\mu(y) = \int_{S^{n-1}}\int_{S^{n-1}}F(\inprod{x,y}) d\mu(x) d\mu(y),
	\]
	we can conclude that $E_\delta(\mu)$ is uniquely maximized by the uniform measure.
\end{proof}

\begin{remark}
	In their paper \cite{Bilyk2016a}, Bilyk et al. remarked that while the Euclidean and geodesic distances are both metrics on the sphere, the phase transition for the behavior of their energy integrals is different. In the Euclidean case, Bj{\"o}rck's theorem shows that it occurs at $\delta = 2$, while in the geodesic case, Bilyk et al.'s theorem shows that it occurs at $\delta = 1$. This peculiar phenomenon is explained by our unified proof of both results.
	
	In both cases, the existence of a phase transition as we let $\delta$ decrease to $0$ is asserted by Lemma \ref{lem: sign of derivatives} and Theorem \ref{thm: energy minimization for general functions}. If the integrand satisfies the hypotheses of Lemma \ref{lem: sign of derivatives} for some $\delta_0$, then for all $0 < \delta < \delta_0$, the integrand will satisfy the hypothesis of Theorem \ref{thm: energy minimization for general functions}, from which we can conclude that the unique maximizer is the uniform measure. For the Euclidean integral, we have $\delta_0 = 2$, while for the geodesic integral, we have $\delta_0 = 1$.
\end{remark}

\begin{remark}
	Bilyk et al. were also interested in understanding continuous functions $F$ for which the uniform measure $\sigma$ is the unique minimizer of $I_F$ as defined in \eqref{energyintegral}. They managed to characterize these functions as those for which all non-constant Gegenbauer coefficients are strictly positive, i.e.
	\[
	\hat{F}(k,\lambda) > 0
	\]
	for all positive integers $k$, and where $\lambda = \frac{n}{2}-1$. On the other hand, by flipping signs, Theorem \ref{thm: energy minimization for general functions} implies that a sufficient condition for this to happen is to require all non-constant Taylor series coefficents to be strictly positive. 
\end{remark}

\section{Discussion}

After submitting the first version of this paper, I became aware that a partial version of Corollary \ref{cor: spherical minimization} was proved by Ehler and Okoudjou in \cite{Ehler2012} (see Theorem 4.10 therein). Their result gives the inequality portion of the corollary but not the uniqueness part of it. They also do not prove any other part of Theorem \ref{thm: minimization and uniqueness theorem}, which applies to more general random vectors, and for all positive integer moments (as opposed to just even integer moments).

Like Bilyk et al., Ehler and Okoudjou obtained their result using spherical harmonics, and in particular, by considering the Gegenbauer coefficients of monomial functions. This is more evidence that there should be a close relationship between the theory of eccentricity tensors and that of spherical harmonics, and it will be interesting to investigate this connection further.

\section*{Acknowledgements}

I would like to thank Dmitriy Bilyk, Han Huang, and Roman Vershynin for insightful discussions on these topics. I would also like to thank the anonymous reviewer for their helpful feedback.

\nocite{*}
\bibliographystyle{amsplain}
\bibliography{minimization_bib}

\end{document}